\documentclass[11pt]{amsart}
\usepackage{amsmath,amssymb}
\newtheorem{theorem}{Theorem}[section]
\newtheorem{proposition}[theorem]{Proposition}
\newtheorem{corollary}[theorem]{Corollary}
\newtheorem{lemma}[theorem]{Lemma}
\theoremstyle{definition}
\newtheorem{definition}[theorem]{Definition}
\newtheorem{remark}[theorem]{Remark}

\begin{document} 

\title[Spacetime positive energy theorem in arbitrary dimension]{On the spacetime positive energy theorem in arbitrary dimension}
\author{Simon Brendle and Yipeng Wang}
\address{Columbia University \\ 2990 Broadway \\ New York NY 10027 \\ USA}
\address{Columbia University \\ 2990 Broadway \\ New York NY 10027 \\ USA}
\thanks{The first author was supported by the National Science Foundation under grant DMS-2403981 and by the Simons Foundation.}
\begin{abstract}
We describe how the spacetime positive energy theorem in dimension $n \geq 4$ follows from our recent work on the Riemannian version of the positive mass theorem. Our work builds on the fundamental work of Schoen and Yau in dimension $n=3$ and its extension to dimension $3 \leq n \leq 7$ in the remarkable work of Eichmair. The proof relies on the Jang equation with a capillary term. We also use the shielding principle from the work of Lesourd-Unger-Yau.
\end{abstract}
\maketitle

\section{Introduction}

Let $(M,g)$ be a Riemannian manifold of dimension $n \geq 4$, and let $q$ be a symmetric $(0,2)$-tensor on $M$. We define a scalar function $\mu$ on $M$ by 
\[\mu = \frac{1}{2} \, (R_g - |q|_g^2 + \text{\rm tr}_g(q)^2).\] 
We define a one-form $J$ on $M$ by 
\[J_k = g^{ij} \, (D_i q_{jk} - D_k q_{ij}) = g^{ij} \, D_i q_{jk} - \partial_k \text{\rm tr}_g(q).\] 

\begin{theorem}
\label{main.theorem}
Suppose that $(M,g)$ is a Riemannian manifold of dimension $n \geq 4$, and $q$ is a symmetric $(0,2)$-tensor on $M$. We assume that there exists a compact subset $K \subset M$ such that $M \setminus K$ is diffeomorphic to the complement of a ball in $\mathbb{R}^n$. Moreover, we assume that there exist real numbers $\alpha$ and $\delta>0$ such that 
\[|\bar{D}^m(g - (1+\alpha \, r^{2-n}) \, \bar{g})|_{\bar{g}} \leq C(m) \, r^{2-n-m-2\delta}\] 
and 
\[|\bar{D}^m q|_{\bar{g}} \leq C(m) \, r^{1-n-m}\] 
for all points near infinity and for every nonnegative integer $m$. Here, $\bar{g}$ denotes the Euclidean metric, $\bar{D}^m$ denotes the covariant derivative of order $m$ with respect to $\bar{g}$, and $r = \sqrt{x_1^2+\hdots+x_n^2}$ denotes the radial coordinate near infinity. Finally, we assume that $\mu - |J|_g > 0$ at each point in $M$. Then $\alpha > 0$.
\end{theorem}

The spacetime positive energy theorem in dimension $3$ was proved in a famous work of Schoen and Yau \cite{Schoen-Yau3}. In their proof, Schoen and Yau used the Jang equation (see \cite{Jang}) to reduce the spacetime positive energy theorem to the Riemannian positive mass theorem \cite{Schoen-Yau1},\cite{Schoen-Yau2}. The spacetime version of the positive energy theorem in dimension $3 \leq n \leq 7$ was proved by Eichmair \cite{Eichmair3} using the Jang equation, and by Eichmair-Huang-Lee-Schoen \cite{Eichmair-Huang-Lee-Schoen} using marginally outer trapped surfaces. The Jang equation was studied further in \cite{Eichmair1}, \cite{Eichmair2}, and \cite{Eichmair-Metzger}. Finally, Lohkamp \cite{Lohkamp} has proposed a proof of the spacetime positive energy theorem in all dimensions.

Our proof of Theorem \ref{main.theorem} relies on a modified Jang equation that includes a capillary term. This modified Jang equation is similar to the one introduced by Schoen and Yau \cite{Schoen-Yau3}. The latter equation also plays an important role in Eichmair's work \cite{Eichmair3}. 

\section{Proof of Theorem \ref{main.theorem}}

Suppose that $(M,g,q)$ is a triplet satisfying the assumptions of Theorem \ref{main.theorem}. We claim that $\alpha > 0$. To prove this, we argue by contradiction. Suppose that $\alpha \leq 0$. 

\subsection{A barrier construction from the work of Schoen-Yau and Eichmair} In this subsection, we recall a barrier construction for the Jang equation. This construction goes back to the fundamental work of Schoen and Yau \cite{Schoen-Yau3} in the three-dimensional case, and was later extended by Eichmair \cite{Eichmair3} to the higher-dimensional case. In this paper, we focus on the case $n \geq 4$ treated in Eichmair's work \cite{Eichmair3}.

\begin{proposition}[R.~Schoen, S.T.~Yau \cite{Schoen-Yau3}; M.~Eichmair \cite{Eichmair3}]
\label{barrier}
If $r_0$ is sufficiently large, then the following holds. We define a positive function $b: (r_0,\infty) \to (0,\infty)$ by 
\[b(s) = r_0 \int_{r_0^{-1} \, s}^\infty (t^{2n-4}-1)^{-\frac{1}{2}} \, dt\] 
for $s \in (r_0,\infty)$. Moreover, we define a positive function $\Upsilon$ on the domain $\{r > r_0\}$ by $\Upsilon = b \circ r$, where $r$ denotes the radial coordinate near infinity. Then 
\[g^{ik} \, g^{jl} \, \big ( g_{ij} - (1+|d\Upsilon|_g^2)^{-1} \, \partial_i \Upsilon \, \partial_j \Upsilon \big ) \, \big ( (1+|d\Upsilon|_g^2)^{-\frac{1}{2}} \, D_{k,l}^2 \Upsilon - q_{kl} \big ) < 0\] 
and 
\[g^{ik} \, g^{jl} \, \big ( g_{ij} - (1+|d\Upsilon|_g^2)^{-1} \, \partial_i \Upsilon \, \partial_j \Upsilon \big ) \, \big ( (1+|d\Upsilon|_g^2)^{-\frac{1}{2}} \, D_{k,l}^2 \Upsilon + q_{kl} \big ) < 0\] 
on the domain $\{r > r_0\}$. 
\end{proposition}

\textbf{Proof.} 
The function $b$ satisfies the ODE 
\begin{equation} 
\label{ode.for.b}
s \, b''(s) + (n-1) \, (1+b'(s)^2) \, b'(s) = -(1+b'(s)^2)^{\frac{3}{2}} \, r_0^{n-2} \, s^{2-n} 
\end{equation}
for $s \in (r_0,\infty)$. The function $\Upsilon = b \circ r$ satisfies 
\begin{align*}
&g^{ik} \, g^{jl} \, \big ( g_{ij} - (1+|d\Upsilon|_g^2)^{-1} \, \partial_i \Upsilon \, \partial_j \Upsilon \big ) \, (1+|d\Upsilon|_g^2)^{-\frac{1}{2}} \, D_{k,l}^2 \Upsilon \\ 
&= (1+b'(r)^2 \, |dr|_g^2)^{-\frac{3}{2}} \, b''(r) \, |dr|_g^2 + (n-1) \, (1+b'(r)^2 \, |dr|_g^2)^{-\frac{1}{2}} \, r^{-1} \, b'(r) \, |dr|_g^2 \\ 
&+ (1+b'(r)^2 \, |dr|_g^2)^{-\frac{1}{2}} \, r^{-1} \, b'(r) \, \big ( g_{ij} - (1+b'(r)^2 \, |dr|_g^2)^{-1} \, b'(r)^2 \, \partial_i r \, \partial_j r \big ) \, T^{ij}, 
\end{align*} 
where 
\[T^{ij} = g^{ik} \, g^{jl} \, (r \, D_{k,l}^2 r + \partial_k r \, \partial_l r - |dr|_g^2 \, g_{kl}).\] 
Using (\ref{ode.for.b}), we obtain  
\begin{align*}
&g^{ik} \, g^{jl} \, \big ( g_{ij} - (1+|d\Upsilon|_g^2)^{-1} \, \partial_i \Upsilon \, \partial_j \Upsilon \big ) \, (1+|d\Upsilon|_g^2)^{-\frac{1}{2}} \, D_{k,l}^2 \Upsilon \\ 
&= -(1+b'(r)^2 \, |dr|_g^2)^{-\frac{3}{2}} \, (1+b'(r)^2)^{\frac{3}{2}} \, r_0^{n-2} \, r^{1-n} \, |dr|_g^2 \\ 
&+ (n-1) \, (1+b'(r)^2 \, |dr|_g^2)^{-\frac{3}{2}} \, r^{-1} \, b'(r)^3 \, |dr|_g^2 \, (|dr|_g^2-1) \\ 
&+ (1+b'(r)^2 \, |dr|_g^2)^{-\frac{1}{2}} \, r^{-1} \, b'(r) \, \big ( g_{ij} - b'(r)^2 \, (1+b'(r)^2 \, |dr|_g^2)^{-1} \, \partial_i r \, \partial_j r \big ) \, T^{ij}. 
\end{align*} 
Using the inequalities 
\[(1+b'(r)^2 \, |dr|_g^2)^{-\frac{3}{2}} \, (1+b'(r)^2)^{\frac{3}{2}} \geq \min \{1,|dr|_g^{-3}\},\] 
\[(1+b'(r)^2 \, |dr|_g^2)^{-1} \, b'(r)^2 \, |dr|_g^2 \leq 1,\] 
and 
\[\big | \big ( g_{ij} - b'(r)^2 \, (1+b'(r)^2 \, |dr|_g^2)^{-1} \, \partial_i r \, \partial_j r \big ) \, T^{ij} \big | \leq n \, |T|_g,\] 
we conclude that 
\begin{align*}
&g^{ik} \, g^{jl} \, \big ( g_{ij} - (1+|d\Upsilon|_g^2)^{-1} \, \partial_i \Upsilon \, \partial_j \Upsilon \big ) \, (1+|d\Upsilon|_g^2)^{-\frac{1}{2}} \, D_{k,l}^2 \Upsilon \\ 
&\leq -r_0^{n-2} \, r^{1-n} \, \min \{|dr|_g^2,|dr|_g^{-1}\} \\ 
&+ (n-1) \, r^{-1} \, |dr|_g^{-1} \, \big | |dr|_g^2-1 \big | + r^{-1} \, |dr|_g^{-1} \, n \, |T|_g. 
\end{align*} 
By assumption, $\big | |dr|_g^2 - 1 \big | \leq C \, r^{2-n}$ and $|T|_g \leq C \, r^{2-n}$. This implies 
\begin{align*}
&g^{ik} \, g^{jl} \, \big ( g_{ij} - (1+|d\Upsilon|_g^2)^{-1} \, \partial_i \Upsilon \, \partial_j \Upsilon \big ) \, (1+|d\Upsilon|_g^2)^{-\frac{1}{2}} \, D_{k,l}^2 \Upsilon \\ 
&\leq -r_0^{n-2} \, r^{1-n} \, \min \{|dr|_g^2,|dr|_g^{-1}\} + C \, r^{1-n}. 
\end{align*} 
Finally, 
\[\big | g^{ik} \, g^{jl} \, \big ( g_{ij} - (1+|d\Upsilon|_g^2)^{-1} \, \partial_i \Upsilon \, \partial_j \Upsilon \big ) \, q_{kl} \big | \leq n \, |q|_g \leq C \, r^{1-n}.\] 
Putting these facts together, the assertion follows. This completes the proof of Proposition \ref{barrier}. \\

In the following, we fix $r_0$ so that the conclusion of Proposition \ref{barrier} holds. 

\begin{proposition}
\label{bound.for.b}
We have $b(s) \leq 2 \, r_0^{n-2} \, s^{3-n}$ for $s \in (2r_0,\infty)$. 
\end{proposition}

\textbf{Proof.} 
Note that $t^{2n-4} - 1 \geq (1 - 2^{4-2n}) \, t^{2n-4}$ for $t \in (2,\infty)$. This implies 
\begin{align*} 
b(s) 
&= r_0 \int_{r_0^{-1} \, s}^\infty (t^{2n-4}-1)^{-\frac{1}{2}} \, dt \\ 
&\leq (1 - 2^{4-2n})^{-\frac{1}{2}} \, r_0 \int_{r_0^{-1} \, s}^\infty t^{2-n} \, dt \\ 
&= (n-3)^{-1} \, (1 - 2^{4-2n})^{-\frac{1}{2}} \, r_0^{n-2} \, s^{3-n} 
\end{align*}
for $s \in (2r_0,\infty)$. Finally, $(n-3)^{-1} \, (1 - 2^{4-2n})^{-\frac{1}{2}} \leq 2$ since $n \geq 4$. This completes the proof of Proposition \ref{bound.for.b}. \\

\subsection{The choice of $\kappa_0$, $\kappa_1$, $s_0$, $s_1$, and $\tau$} In this subsection, we fix various parameters that will be needed in the subsequent arguments. Let $r_0$ be chosen as in Proposition \ref{barrier}. We fix a smooth function $\zeta: M \to [0,1]$ with the property that $\zeta = 0$ on $\{r > 8r_0\}$ and $\zeta = 1$ on $M \setminus \{r > 4r_0\}$. By assumption, $\mu - |J|_g > 0$ at each point in $M$. Since the function $\zeta$ has compact support, we can find positive constants $\kappa_0$ and $\kappa_1$ such that 
\[\mu - |J|_g - \kappa_0^2 \, |d\zeta|_g^2 - \kappa_1 \, \zeta^2 \, n \, |q|_g > 0\] 
at each point in $M$. Consequently, we can find a positive smooth function $Q$ on $M$ such that 
\begin{equation} 
\label{choice.of.Q}
\mu - |J|_g - \kappa_0^2 \, |d\zeta|_g^2 - \kappa_1 \, \zeta^2 \, n \, |q|_g \geq Q 
\end{equation} 
at each point in $M$ and 
\begin{equation} 
\label{decay.of.Q}
|\bar{D}^m Q|_{\bar{g}} \leq C(m) \, r^{-n-m-2\delta} 
\end{equation}
for every nonnegative integer $m$. 

We next define $E_0 = \{r > 8r_0\}$. Since the function $Q$ is strictly positive at each point in $M$, we can find positive real numbers $s_0$ and $s_1$ such that $s_1 \geq s_0$ and 
\begin{equation} 
\label{choice.of.s_0.and.s_1}
Q(x) > \frac{128}{s_1s_0} 
\end{equation}
for each point $x \in \mathcal{N}_g(E_0,2s_0) \setminus E_0$. Having chosen $s_0$ and $s_1$ in this way, we fix $\tau>0$ sufficiently small so that 
\begin{equation} 
\label{tau} 
2^{10-3n} \, r_0 + s_1+2s_0 \leq \frac{1}{2} \, \min \{\kappa_0 \, \tau^{-1},\kappa_1 \, \tau^{-2}\}. 
\end{equation}

\subsection{A modified Jang equation} In this subsection, we construct a solution of a modified Jang equation. To that end, we consider a sequence of real numbers $r_j$ such that $r_j > 32r_0$ for each $j$ and $r_j \to \infty$ as $j \to \infty$. For each $j$, we define a compact domain $M^{(j)} \subset M$ by $M^{(j)} = M \setminus \{r > r_j\}$. We consider the following boundary value problem on $M^{(j)}$, which is inspired by work of Jang \cite{Jang}, Schoen-Yau \cite{Schoen-Yau3}, and Eichmair \cite{Eichmair3}.

\begin{definition}
Let $w$ be a smooth function on $M^{(j)}$, and let $\lambda \in [-1,1]$. We say that $w$ is a solution of $(\star_{\lambda,j})$ if $w$ solves the PDE 
\[g^{ik} \, g^{jl} \, \big ( g_{ij} - (1+|dw|_g^2)^{-1} \, \partial_i w \, \partial_j w \big ) \, \big ( (1+|dw|_g^2)^{-\frac{1}{2}} \, D_{k,l}^2 w - \lambda \, q_{kl} \big ) = \tau^2 \, \zeta^2 \, w\] 
in $M^{(j)}$ with Dirichlet boundary condition $w = 0$ on $\partial M^{(j)}$.
\end{definition}

The PDE above is different from the PDE studied in \cite{Schoen-Yau3} and \cite{Eichmair3} in that we include the cutoff function $\zeta^2$ on the right hand side of the equation. 

In the following, we prove the existence of a solution of $(\star_{1,j})$. The proof is standard, and follows the arguments of Schoen-Yau \cite{Schoen-Yau3}, Eichmair \cite{Eichmair1}, Eichmair-Metzger \cite{Eichmair-Metzger}, and Korevaar-Simon \cite{Korevaar-Simon}. We first establish $C^0$-estimates for solutions of $(\star_{\lambda,j})$.

\begin{lemma}
\label{C0.estimate.version.1}
Suppose that $w$ is a solution of $(\star_{\lambda,j})$ for some $\lambda \in [-1,1]$. Then $|w| \leq b(r) - b(r_j)$ on the domain $\{r_0 < r \leq r_j\}$. 
\end{lemma}

\textbf{Proof.} 
As above, we define a positive function $\Upsilon$ on the domain $\{r > r_0\}$ by $\Upsilon = b \circ r$, where $r$ denotes the radial coordinate near infinity. It suffices to show that 
\begin{equation}
\label{upper.bound.for.w-Upsilon}
\sup_{\{r_0 < r \leq r_j\}} (w-\Upsilon) \leq -b(r_j) 
\end{equation} 
and 
\begin{equation}
\label{lower.bound.for.w+Upsilon}
\inf_{\{r_0 < r \leq r_j\}} (w+\Upsilon) \geq b(r_j). 
\end{equation}
To prove the inequality (\ref{upper.bound.for.w-Upsilon}), we argue by contradiction. Suppose that 
\[\sup_{\{r_0 < r \leq r_j\}} (w-\Upsilon) > -b(r_j).\] 
Since $\lim_{s \searrow r_0} b'(s) = -\infty$, we can find a point $\bar{x} \in \{r_0 < r \leq r_j\}$ such that 
\begin{equation} 
\label{maximum.of.w-Upsilon}
w(\bar{x}) - \Upsilon(\bar{x}) = \sup_{\{r_0 < r \leq r_j\}} (w-\Upsilon) > -b(r_j). 
\end{equation}
Using (\ref{maximum.of.w-Upsilon}) and the fact that $w = 0$ on $\partial M^{(j)} = \{r = r_j\}$, we deduce that $\bar{x} \in \{r_0 < r < r_j\}$. Consequently, $dw = d\Upsilon$ and $D^2 w \leq D^2 \Upsilon$ at the point $\bar{x}$. This implies 
\begin{align*}
&g^{ik} \, g^{jl} \, \big ( g_{ij} - (1+|dw|_g^2)^{-1} \, \partial_i w \, \partial_j w \big ) \, \big ( (1+|dw|_g^2)^{-\frac{1}{2}} \, D_{k,l}^2 w - \lambda \, q_{kl} \big ) \\ 
&\leq g^{ik} \, g^{jl} \, \big ( g_{ij} - (1+|d\Upsilon|_g^2)^{-1} \, \partial_i \Upsilon \, \partial_j \Upsilon \big ) \, \big ( (1+|d\Upsilon|_g^2)^{-\frac{1}{2}} \, D_{k,l}^2 \Upsilon - \lambda \, q_{kl} \big )
\end{align*} 
at the point $\bar{x}$. The expression on the right hand side is strictly negative by Proposition \ref{barrier}. Since $w$ is a solution of $(\star_{\lambda,j})$, it follows that $\tau^2 \, \zeta^2 \, w(\bar{x}) < 0$. Therefore, $w(\bar{x}) < 0$. Since $\Upsilon(\bar{x}) \geq b(r_j)$, we conclude that $w(\bar{x}) - \Upsilon(\bar{x}) < -b(r_j)$. This contradicts (\ref{maximum.of.w-Upsilon}). This proves (\ref{upper.bound.for.w-Upsilon}). The inequality (\ref{lower.bound.for.w+Upsilon}) follows by replacing $w$ by $-w$ and $\lambda$ by $-\lambda$. This completes the proof of Lemma \ref{C0.estimate.version.1}. \\

\begin{remark}
The estimate $|w| \leq b(r) - b(r_j)$ in Lemma \ref{C0.estimate.version.1} holds with equality along the boundary $\partial M^{(j)}$. In particular, $|dw|_g \leq |b'(r_j)| \, |dr|_g$ at each point on the boundary $\partial M^{(j)}$. 
\end{remark}

\begin{lemma}
\label{C0.estimate.version.2}
Suppose that $w$ is a solution of $(\star_{\lambda,j})$ for some $\lambda \in [-1,1]$. Then $|w| \leq 2 \, r_0^{n-2} \, r^{3-n}$ on the domain $\{2r_0 < r \leq r_j\}$. 
\end{lemma}

\textbf{Proof.} 
This follows by combining Proposition \ref{bound.for.b} and Lemma \ref{C0.estimate.version.1}. \\

\begin{lemma}
\label{C0.estimate.version.3} 
Suppose that $w$ is a solution of $(\star_{\lambda,j})$ for some $\lambda \in [-1,1]$. Then 
\[\sup_{M^{(j)}} |w| \leq \max \Big \{ 2^{4-n} \, r_0,\tau^{-2} \sup_M n \, |q|_g \Big \}.\]
\end{lemma}

\textbf{Proof.} 
It suffices to show that 
\begin{equation} 
\label{upper.bound.for.w}
\sup_{M^{(j)}} w \leq \max \Big \{ 2^{4-n} \, r_0,\tau^{-2} \sup_M n \, |q|_g \Big \}. 
\end{equation} 
and 
\begin{equation} 
\label{lower.bound.for.w}
\inf_{M^{(j)}} w \geq -\max \Big \{ 2^{4-n} \, r_0,\tau^{-2} \sup_M n \, |q|_g \Big \}. 
\end{equation} 
To prove the inequality (\ref{upper.bound.for.w}), we argue by contradiction. Suppose that 
\[\sup_{M^{(j)}} w > \max \Big \{ 2^{4-n} \, r_0,\tau^{-2} \sup_M n \, |q|_g \Big \}.\] 
We can find a point $\bar{x} \in M^{(j)}$ such that 
\begin{equation} 
\label{maximum.of.w}
w(\bar{x}) = \sup_{M^{(j)}} w > \max \Big \{ 2^{4-n} \, r_0,\tau^{-2} \sup_M n \, |q|_g \Big \}. 
\end{equation} 
Using (\ref{maximum.of.w}) and Lemma \ref{C0.estimate.version.2}, we deduce that $\bar{x} \in M^{(j)} \setminus \{2r_0 < r \leq r_j\}$. Consequently, $\zeta = 1$ at the point $\bar{x}$. Moreover, $dw = 0$ and $D^2 w \leq 0$ at the point $\bar{x}$. This implies 
\[g^{ik} \, g^{jl} \, \big ( g_{ij} - (1+|dw|_g^2)^{-1} \, \partial_i w \, \partial_j w \big ) \, \big ( (1+|dw|_g^2)^{-\frac{1}{2}} \, D_{k,l}^2 w - \lambda \, q_{kl} \big ) \leq n \, |q|_g\] 
at the point $\bar{x}$. Since $w$ is a solution of $(\star_{\lambda,j})$, it follows that $\tau^2 \, w \leq n \, |q|_g$ at the point $\bar{x}$. This contradicts (\ref{maximum.of.w}). This proves (\ref{upper.bound.for.w}). The inequality (\ref{lower.bound.for.w}) follows by replacing $w$ by $-w$ and $\lambda$ by $-\lambda$. This completes the proof of Lemma \ref{C0.estimate.version.3}. \\

In the next step, we establish $C^1$-estimates for solutions of $(\star_{\lambda,j})$. 

\begin{lemma}
\label{gradient.estimate.outer.region}
Suppose that $w$ is a solution of $(\star_{\lambda,j})$ for some $\lambda \in [-1,1]$. Then $|dw|_g \leq C$ in the region $\{\frac{1}{2} \, r_j < r \leq r_j\}$. Here, $C$ is a large constant that is independent of $j$.
\end{lemma}

\textbf{Proof.} 
Let $d_{\bar{g}}$ denote the distance function with respect to the Euclidean metric $\bar{g}$. Given a point $p$ in the region $\{\frac{1}{2} \, r_j < r < r_j\}$, we consider the Euclidean ball $\mathcal{B}_{\bar{g}}(p,\sigma)$, where $\sigma = \frac{1}{2} \, d_{\bar{g}}(p,\partial M^{(j)}) = \frac{1}{2} \, (r_j - r(p))$. Lemma \ref{C0.estimate.version.1} implies that $\sup_{\mathcal{B}_{\bar{g}}(p,\sigma)} |w| \leq C_0 \sigma$ for some uniform constant $C_0$. We now apply Proposition \ref{interior.gradient.estimate} with $\Omega = \mathcal{B}_{\bar{g}}(p,\sigma)$ and 
\[\psi = 2C_0 \sigma^{-1} \, (2 \, d_{\bar{g}}(p,\cdot)^2 - \sigma^2).\] 
Thus, we conclude that $|dw|_g \leq C$ at the point $p$, where $C$ is a large constant that is independent of $\sigma$ and $j$. This completes the proof of Lemma \ref{gradient.estimate.outer.region}. \\

\begin{lemma}
\label{gradient.estimate.inner.region}
Suppose that $w$ is a solution of $(\star_{\lambda,j})$ for some $\lambda \in [-1,1]$. Then $|dw|_g \leq C$ in the region $M^{(j)} \setminus \{\frac{1}{2} \, r_j < r \leq r_j\}$. Here, $C$ is a large constant that is independent of $j$.
\end{lemma}

\textbf{Proof.}
Let us fix a positive real number $\sigma$ such that $\sigma < \frac{1}{2} \, \text{\rm inj}(M,g)$. It follows from Lemma \ref{C0.estimate.version.3} that $\sup_{M^{(j)}} |w| \leq C_1 \sigma$, where 
\[C_1 = \sigma^{-1} \, \max \Big \{ 2^{4-n} \, r_0,\tau^{-2} \sup_M n \, |q|_g \Big \}.\]
We now apply Proposition \ref{interior.gradient.estimate} with $\Omega = \mathcal{B}_g(p,\sigma)$ and 
\[\psi = 2C_1 \sigma^{-1} \, (2 \, d_g(p,\cdot)^2 - \sigma^2).\] 
Thus, we conclude that $|dw|_g \leq C$ at the point $p$, where $C$ is a large constant that is independent of $j$. This completes the proof of Lemma \ref{gradient.estimate.inner.region}. \\

\begin{proposition}
\label{existence.of.a.solution.for.each.j}
For each $j$, the boundary value problem $(\star_{1,j})$ has a smooth solution $u^{(j)}$. Moreover, $|u^{(j)}| \leq 2 \, r_0^{n-2} \, r^{3-n}$ on the domain $\{2r_0 < r \leq r_j\}$. 
\end{proposition}

\textbf{Proof.} 
Let us fix an integer $j$. Lemma \ref{C0.estimate.version.3} gives a $C^0$-estimate for solutions of $(\star_{\lambda,j})$. Using Lemma \ref{gradient.estimate.outer.region} and Lemma \ref{gradient.estimate.inner.region}, we obtain a $C^1$-estimate up to the boundary of $M^{(j)}$ for solutions of $(\star_{\lambda,j})$. Standard elliptic theory (see e.g. Theorem 13.7 in \cite{Gilbarg-Trudinger}) gives bounds for the higher derivatives of solutions of $(\star_{\lambda,j})$. Therefore, the set 
\[\{\lambda \in [-1,1]: \text{\rm $(\star_{\lambda,j})$ has a solution}\}\] 
is a closed subset of $[-1,1]$. On the other hand, it follows from Theorem 6.14 in \cite{Gilbarg-Trudinger} that, for every solution of $(\star_{\lambda,j})$, the linearized operator is invertible. Consequently, the set 
\[\{\lambda \in [-1,1]: \text{\rm $(\star_{\lambda,j})$ has a solution}\}\] 
is an open subset of $[-1,1]$. Since the function $w=0$ is a solution of $(\star_{0,j})$, we conclude that $(\star_{\lambda,j})$ has a solution for each $\lambda \in [-1,1]$. This completes the proof of Proposition \ref{existence.of.a.solution.for.each.j}. \\

\begin{proposition}
\label{existence}
We can find a smooth function $u$ on $M$ such that 
\[g^{ik} \, g^{jl} \, \big ( g_{ij} - (1+|du|_g^2)^{-1} \, \partial_i u \, \partial_j u \big ) \, \big ( (1+|du|_g^2)^{-\frac{1}{2}} \, D_{k,l}^2 u - q_{kl} \big ) = \tau^2 \, \zeta^2 \, u\] 
at each point in $M$ and $|u| \leq 2 \, r_0^{n-2} \, r^{3-n}$ on the domain $\{r > 2r_0\}$.
\end{proposition}

\textbf{Proof.} 
By Lemma \ref{C0.estimate.version.3}, the sequence $u^{(j)}$ is uniformly bounded in $C^0$. Using Lemma \ref{gradient.estimate.inner.region}, we obtain uniform $C^1$-bounds on every compact subset of $M$. Standard elliptic theory (see e.g. Theorem 13.6 in \cite{Gilbarg-Trudinger}) now gives uniform bounds for the higher derivatives of $u^{(j)}$ on every compact subset of $M$. Hence, after passing to a subsequence if necessary, the sequence $u^{(j)}$ converges in $C_{\text{\rm loc}}^\infty$ to a smooth function $u$. It is easy to see that $u$ has all the required properties. \\

\begin{proposition}
\label{higher.derivative.bounds}
For every positive integer $m$, the $m$-th order derivatives of the function $u$ satisfy the bound  
\[|\bar{D}^m u|_{\bar{g}} \leq C(m) \, r^{3-n-m}\] 
on the domain $\{r > 32r_0\}$. 
\end{proposition}

\textbf{Proof.}
It follows from Lemma \ref{gradient.estimate.inner.region} that $|du|_g \leq C$ for some uniform constant $C$. Standard elliptic theory (see e.g. Theorem 13.6 in \cite{Gilbarg-Trudinger}) now implies that 
\[|\bar{D}^m u|_{\bar{g}} \leq C(m) \, r^{1-m}\] 
on the domain $\{r>8r_0\}$. On the domain $\{r > 16r_0\}$, we have $|u| \leq 2 \, r_0^{n-2} \, r^{3-n}$ and $|\bar{D}^m q|_{\bar{g}} \leq C(m) \, r^{1-n-m}$ for every nonnegative integer $m$. Hence, standard interior estimates for linear PDE imply that 
\[|\bar{D}^m u|_{\bar{g}} \leq C(m) \, r^{3-n-m}\] 
on the domain $\{r > 32r_0\}$. This completes the proof of Proposition \ref{higher.derivative.bounds}. \\

\subsection{The metric $\check{g}$} Let $u$ denote the function constructed in Proposition \ref{existence}. We define a smooth Riemannian metric $\check{g}$ and a $1$-form $\Xi$ on $M$ by 
\[\check{g}_{ij} = g_{ij} + \partial_i u \, \partial_j u\] 
and 
\[\Xi_k = \frac{1}{2} \, \partial_k \log(1+|du|_g^2) - (1+|du|_g^2)^{-\frac{1}{2}} \, g^{ij} \, q_{ik} \, \partial_j u.\] 

\begin{lemma}
\label{decay}
We have 
\[|\bar{D}^m(\check{g} - (1+\alpha \, r^{2-n}) \, \bar{g})|_{\bar{g}} \leq O(r^{2-n-m-2\delta})\] 
for every nonnegative integer $m$. Moreover, $|R_{\check{g}}| \leq O(r^{-n-2\delta})$, $|\Xi|_{\check{g}} \leq O(r^{3-2n})$, and $|\text{\rm div}_{\check{g}} \Xi| \leq O(r^{2-2n})$. In particular, the functions $|R_{\check{g}}|$, $|\Xi|_{\check{g}}^2$, and $|\text{\rm div}_{\check{g}} \Xi|$ are integrable. 
\end{lemma}

\textbf{Proof.}
Using Proposition \ref{higher.derivative.bounds} together with the identity $\check{g} - g = du \otimes du$, we obtain 
\[|\bar{D}^m(\check{g} - g)|_{\bar{g}} \leq O(r^{4-2n-m})\] 
for every nonnegative integer $m$. Moreover, 
\[|\bar{D}^m(g - (1+\alpha \, r^{2-n}) \, \bar{g})|_{\bar{g}} \leq O(r^{2-n-m-2\delta})\] 
for every nonnegative integer $m$. Putting these facts together, we conclude that 
\[|\bar{D}^m(\check{g} - (1+\alpha \, r^{2-n}) \, \bar{g})|_{\bar{g}} \leq O(r^{2-n-m-2\delta})\] 
for every nonnegative integer $m$. This implies $|R_{\check{g}}| \leq O(r^{-n-2\delta})$. Moreover, using Proposition \ref{higher.derivative.bounds} and the estimate $|q|_{\bar{g}} \leq O(r^{1-n})$, we obtain $|\Xi|_{\bar{g}} \leq O(r^{3-2n})$ near infinity. The estimate $|\text{\rm div}_{\check{g}} \Xi| \leq O(r^{2-2n})$ follows similarly. This completes the proof of Lemma \ref{decay}. \\

\begin{proposition}
\label{consequence.of.Schoen.Yau.identity}
We have 
\[\frac{1}{2} \, R_{\check{g}} - |\Xi|_{\check{g}}^2 + \text{\rm div}_{\check{g}} \Xi \geq Q + (\kappa_0^2 - \tau^2 \, u^2) \, |d\zeta|_g^2 + (\kappa_1 - \tau^2 \, |u|) \, \zeta^2 \, n \, |q|_g\] 
at each point in $M$.
\end{proposition}

\textbf{Proof.}
Applying Proposition \ref{Schoen.Yau.identity} with $w = u$ and $\Theta = \tau^2 \, \zeta^2 \, u$ gives 
\begin{align*} 
&\frac{1}{2} \, R_{\check{g}} - |\Xi|_{\check{g}}^2 + \text{\rm div}_{\check{g}} \Xi \\ 
&= \frac{1}{2} \, \check{g}^{ik} \, \check{g}^{jl} \, \big ( (1+|du|_g^2)^{-\frac{1}{2}} \, D_{i,j}^2 u - q_{ij} \big ) \, \big ( (1+|du|_g^2)^{-\frac{1}{2}} \, D_{k,l}^2 u - q_{kl} \big ) \\ 
&+ \mu - (1+|du|_g^2)^{-\frac{1}{2}} \, \langle du,J \rangle_g \\ 
&+ \tau^2 \, (1+|du|_g^2)^{-\frac{1}{2}} \, (|\zeta \, du + u \, d\zeta|_g^2 - u^2 \, |d\zeta|_g^2) + \frac{1}{2} \, \tau^4 \, \zeta^4 \, u^2 + \tau^2 \, \zeta^2 \, u \, \text{\rm tr}_{\check{g}}(q) 
\end{align*} 
at each point in $M$. Using the inequality $|\text{\rm tr}_{\check{g}}(q)| \leq n \, |q|_g$, we obtain
\begin{equation} 
\label{inequality.a}
\frac{1}{2} \, R_{\check{g}} - |\Xi|_{\check{g}}^2 + \text{\rm div}_{\check{g}} \Xi \geq \mu - |J|_g - \tau^2 \, u^2 \, |d\zeta|_g^2 - \tau^2 \, \zeta^2 \, |u| \, n \, |q|_g 
\end{equation} 
at each point in $M$. On the other hand, the inequality (\ref{choice.of.Q}) gives
\begin{align} 
\label{inequality.b}
&\mu - |J|_g - \tau^2 \, u^2 \, |d\zeta|_g^2 - \tau^2 \, \zeta^2 \, |u| \, n \, |q|_g \notag \\ 
&\geq Q + (\kappa_0^2 - \tau^2 \, u^2) \, |d\zeta|_g^2 + (\kappa_1 - \tau^2 \, |u|) \, \zeta^2 \, n \, |q|_g 
\end{align}
at each point in $M$. If we combine (\ref{inequality.a}) and (\ref{inequality.b}), the assertion follows. This completes the proof of Proposition \ref{consequence.of.Schoen.Yau.identity}. \\

\begin{corollary}
\label{quadratic.form}
Suppose that $f$ is a smooth test function with the property that $f$ is supported in the domain $\big \{ |u| < \min \{\kappa_0 \, \tau^{-1},\kappa_1 \, \tau^{-2}\} \big \}$ and $f$ is constant near infinity. Then 
\[\int_M |df|_{\check{g}}^2 \, d\text{\rm vol}_{\check{g}} + \frac{1}{2} \int_M R_{\check{g}} \, f^2 \, d\text{\rm vol}_{\check{g}} \geq \int_M Q \, f^2 \, d\text{\rm vol}_{\check{g}}.\] 
Note that the integral $\int_M R_{\check{g}} \, f^2 \, d\text{\rm vol}_{\check{g}}$ is well-defined in view of Lemma \ref{decay}, and the integral $\int_M Q \, f^2 \, d\text{\rm vol}_{\check{g}}$ is well-defined in view of (\ref{decay.of.Q}).
\end{corollary}

\textbf{Proof.} 
Proposition \ref{consequence.of.Schoen.Yau.identity} implies that 
\[\frac{1}{2} \, R_{\check{g}} - |\Xi|_{\check{g}}^2 + \text{\rm div}_{\check{g}} \Xi \geq Q\] 
on the domain $\big \{ |u| < \min \{\kappa_0 \, \tau^{-1},\kappa_1 \, \tau^{-2}\} \big \}$. Since $f$ is supported in the domain $\big \{ |u| < \min \{\kappa_0 \, \tau^{-1},\kappa_1 \, \tau^{-2}\} \big \}$, we obtain 
\begin{equation} 
\label{integrated.version.of.Schoen.Yau.identity}
\int_M \Big ( \frac{1}{2} \, R_{\check{g}} - |\Xi|_{\check{g}}^2 + \text{\rm div}_{\check{g}} \Xi \Big ) \, f^2 \, d\text{\rm vol}_{\check{g}} \geq \int_M Q \, f^2 \, d\text{\rm vol}_{\check{g}}. 
\end{equation} 
Note that the integral on the left hand side is well-defined in view of Lemma \ref{decay}, and the integral on the right hand side is well-defined in view of (\ref{decay.of.Q}). Moreover, Lemma \ref{decay} implies that $|\Xi|_{\check{g}} \leq O(r^{3-2n})$ near infinity. Since $f$ is constant near infinity and $\check{g}$ is uniformly equivalent to $\bar{g}$ near infinity, we conclude that 
\begin{equation} 
\label{divergence.theorem}
\int_M \text{\rm div}_{\check{g}}(f^2 \, \Xi) \, d\text{\rm vol}_{\check{g}} = 0 
\end{equation} 
by the divergence theorem. Subtracting (\ref{divergence.theorem}) from (\ref{integrated.version.of.Schoen.Yau.identity}) gives 
\begin{align*}  
&\int_M |df|_{\check{g}}^2 \, d\text{\rm vol}_{\check{g}} + \frac{1}{2} \int_M R_{\check{g}} \, f^2 \, d\text{\rm vol}_{\check{g}} - \int_M |df + f \, \Xi|_{\check{g}}^2 \, d\text{\rm vol}_{\check{g}} \\ 
&\geq \int_M Q \, f^2 \, d\text{\rm vol}_{\check{g}}. 
\end{align*} 
This completes the proof of Corollary \ref{quadratic.form}. \\

\begin{lemma}
\label{tubular.neighborhood.1}
We have 
\[\mathcal{N}_{(M,\check{g})}(E_0,s_1+2s_0) \subset \Big \{ |u| \leq \frac{1}{2} \, \min \{\kappa_0 \, \tau^{-1},\kappa_1 \, \tau^{-2}\} \big \}.\] 
\end{lemma}

\textbf{Proof.}
Recall that 
\[|u| \leq 2 \, r_0^{n-2} \, r^{3-n} \leq 2^{10-3n} \, r_0\] 
on the domain $E_0 = \{r > 8r_0\}$. Moreover, $|du|_{\check{g}} \leq 1$ by definition of the metric $\check{g}$. Using (\ref{tau}), we conclude that 
\[|u| \leq 2^{10-3n} \, r_0 + s_1+2s_0 \leq \frac{1}{2} \, \min \{\kappa_0 \, \tau^{-1},\kappa_1 \, \tau^{-2}\}\] 
on the set $\mathcal{N}_{(M,\check{g})}(E_0,s_1+2s_0)$. This completes the proof of Lemma \ref{tubular.neighborhood.1}. \\

\begin{lemma}
\label{tubular.neighborhood.2}
We have 
\[\mathcal{N}_{(M,\check{g})}(E_0,2s_0) \subset \mathcal{N}_g(E_0,2s_0).\] 
In particular, 
\[Q(x) > \frac{128}{s_1s_0}\] 
for each point $x \in \mathcal{N}_{(M,\check{g})}(E_0,2s_0) \setminus E_0$.
\end{lemma}

\textbf{Proof.} 
The first statement follows from the pointwise inequality $\check{g} \geq g$. The second statement follows from the first statement together with (\ref{choice.of.s_0.and.s_1}). This completes the proof of Lemma \ref{tubular.neighborhood.2}. \\

In the next step, we recall an important lemma from the work of Lesourd-Unger-Yau.

\begin{lemma}[cf. Lesourd-Unger-Yau \cite{Lesourd-Unger-Yau}, Proposition 3.1]
\label{Phi}
We can find an open, connected domain $E$ with smooth boundary, a smooth function $\Phi$ defined on $E$, and a smooth function $\hat{Q}$ defined on $E$ with the following properties: 
\begin{itemize}
\item The closure of $E_0$ is contained in $E$.
\item The domain $E$ is contained in $\mathcal{N}_{(M,\check{g})}(E_0,s_1+2s_0)$. In particular, the complement $E \setminus E_0$ is a bounded subset of $(M,g)$. 
\item $\Phi=0$ and $\hat{Q} = \frac{1}{2} \, Q$ at each point in $E_0$.
\item $\Phi \leq 0$ and $\hat{Q} > 0$ at each point in $E$.
\item $\Phi \to -\infty$ on the boundary $\partial E$. 
\item $Q + \frac{1}{2} \, \Phi^2 - 2 \, |d\Phi|_{\check{g}} \geq 2\hat{Q}$ at each point in $E$. 
\end{itemize}
\end{lemma}

Lemma \ref{Phi} follows from Lemma \ref{tubular.neighborhood.2}. A detailed proof can be found in \cite{Brendle-Wang}. 

\begin{remark}
In Lemma \ref{Phi}, we allow the possibility that $\partial E = \emptyset$. 
\end{remark}

After these preparations, we now complete the proof of Theorem \ref{main.theorem}. Using Lemma \ref{tubular.neighborhood.1} and Lemma \ref{Phi}, we obtain 
\[E \subset \mathcal{N}_{(M,\check{g})}(E_0,s_1+2s_0) \subset \Big \{ |u| \leq \frac{1}{2} \, \min \{\kappa_0 \, \tau^{-1},\kappa_1 \, \tau^{-2}\} \big \}\]
by Lemma \ref{tubular.neighborhood.1} and Lemma \ref{Phi}. Consequently, we can find an open, connected domain $\hat{E}$ with smooth boundary such that the closure of $E$ is contained in $\hat{E}$, and the closure of $\hat{E}$ is contained in the domain $\big \{ |u| < \min \{\kappa_0 \, \tau^{-1},\kappa_1 \, \tau^{-2}\} \big \}$. Using Corollary \ref{quadratic.form}, we conclude that 
\[\int_M |df|_{\check{g}}^2 \, d\text{\rm vol}_{\check{g}} + \frac{1}{2} \int_M R_{\check{g}} \, f^2 \, d\text{\rm vol}_{\check{g}} \geq \int_M Q \, f^2 \, d\text{\rm vol}_{\check{g}}\] 
for every smooth test function $f$ with the property that $f$ is supported in $\hat{E}$ and $f$ is constant near infinity. Finally, since $\alpha \leq 0$, Lemma \ref{decay} implies that $(M,\check{g})$ has nonpositive mass. 

We now follow the arguments in Subsections 3.2--3.8 in \cite{Brendle-Wang}, with $\rho=1$. In conclusion, we obtain an $(n-1)$-dataset (in the sense of Definition 1.3 in \cite{Brendle-Wang}) which has zero mass (in the sense of Definition 1.4 in \cite{Brendle-Wang}). This contradicts Theorem 1.5 in \cite{Brendle-Wang}. This completes the proof of Theorem \ref{main.theorem}. 

\appendix 

\section{A Korevaar-Simon-type estimate for the modified Jang equation}

The following estimate goes back to work of Korevaar-Simon \cite{Korevaar-Simon} and was extended to the Jang equation by Eichmair \cite{Eichmair1} and Eichmair-Metzger \cite{Eichmair-Metzger}. 

\begin{proposition}[cf. N.~Korevaar, L.~Simon \cite{Korevaar-Simon}]
\label{interior.gradient.estimate}
Let $A \geq 4$ and $\sigma>0$. Let $(M,g)$ be a Riemannian manifold of dimension $n$, and let $q$ be a symmetric $(0,2)$-tensor on $M$. Let $\Omega$ be an open domain in $M$ with compact closure, and let $\zeta,\psi \in C(\bar{\Omega}) \cap C^\infty(\Omega)$. We assume that $\text{\rm Ric} \geq -A \, \sigma^{-2} \, g$, $n \, |q|_g \leq A \, \sigma^{-1}$, $n \, |Dq|_g \leq A \, \sigma^{-2}$, $|\psi| \leq \frac{1}{4} \, A \, \sigma$, $|d\psi|_g \leq \frac{1}{4} \, A$ and $n \, |D^2 \psi|_g \leq \frac{1}{4} \, A \, \sigma^{-1}$ at each point in $\Omega$. Let $w \in C(\bar{\Omega}) \cap C^\infty(\Omega)$ be a solution of the PDE 
\begin{equation} 
\label{pde.for.w.1}
\big ( g^{ij} - (1+|dw|_g^2)^{-1} \, \partial^i w \, \partial^j w) \, \big ( (1+|dw|_g^2)^{-\frac{1}{2}} \, D_{i,j}^2 w - q_{ij} \big ) = \zeta^2 \, w 
\end{equation}
in $\Omega$ satisfying $w < \psi$ along $\partial \Omega$. If 
\[(1 + \zeta^2 \, \sigma^2 + |d\zeta|_g^2 \, \sigma^4) \, |w| \leq A \, \sigma\] 
at each point in $\Omega$, then 
\[(e^{A^2 \sigma^{-1} (w-\psi)} - 1) \, (1+|dw|_g^2)^{\frac{1}{2}} \leq C(A)\] 
at each point in $\Omega$, where $C(A)$ is a constant that depends only on $A$.
\end{proposition}

\textbf{Proof.} 
Let 
\[\Lambda = \sup_\Omega (e^{A^2 \sigma^{-1} (w-\psi)} - 1) \, (1+|dw|_g^2)^{\frac{1}{2}}.\] 
If $\Lambda \leq 0$, we are done. It remains to consider the case $\Lambda > 0$. Since $w < \psi$ along $\partial \Omega$, we can find a point $\bar{x} \in \Omega$ where the function 
\[(e^{A^2 \sigma^{-1} (w-\psi)} - 1) \, (1+|dw|_g^2)^{\frac{1}{2}}\] 
attains its maximum. Clearly, 
\[e^{A^2 \sigma^{-1} (w-\psi)}-1 \leq \Lambda \, (1+|dw|_g^2)^{-\frac{1}{2}}\] 
at each point in $\Omega$, and equality holds at the point $\bar{x}$. For abbreviation, let 
\[S^{ij} = 2 \, q^{ij} - (1+|dw|_g^2)^{-1} \, q^{kl} \, \partial_k w \, \partial_l w \, g^{ij} - (\text{\rm tr}_g(q) + \zeta^2 \, w) \, g^{ij}.\] 
Then 
\begin{align*} 
&\big ( g^{ij} - (1+|dw|_g^2)^{-1} \, \partial^i w \, \partial^j w \big ) \, D_{i,j}^2(e^{A^2 \sigma^{-1} (w-\psi)}) \\ 
&+ (1+|dw|_g^2)^{-\frac{1}{2}} \, S^{ij} \, \partial_i w \, \partial_j(e^{A^2 \sigma^{-1} (w-\psi)}) \\ 
&\leq \Lambda \, \big ( g^{ij} - (1+|dw|_g^2)^{-1} \, \partial^i w \, \partial^j w \big ) \, D_{i,j}^2 \big ( (1+|dw|_g^2)^{-\frac{1}{2}} \big ) \\ 
&+ \Lambda \, (1+|dw|_g^2)^{-\frac{1}{2}} \, S^{ij} \, \partial_i w \, \partial_j \big ( (1+|dw|_g^2)^{-\frac{1}{2}} \big )
\end{align*}
at the point $\bar{x}$. The identity (\ref{pde.for.w.1}) implies 
\begin{align} 
\label{pde.for.w.2}
&\big ( g^{ij} - (1+|dw|_g^2)^{-1} \, \partial^i w \, \partial^j w) \, D_{i,j}^2 w \notag \\ 
&= -(1+|dw|_g^2)^{-\frac{1}{2}} \, q^{ij} \, \partial_i w \, \partial_j w + (1+|dw|_g^2)^{\frac{1}{2}} \, (\text{\rm tr}_g(q) + \zeta^2 \, w) 
\end{align} 
at each point in $\Omega$. Using the identity 
\begin{align*} 
&\big ( g^{ij} - (1+|dw|_g^2)^{-1} \, \partial^i w \, \partial^j w \big ) \, D_{i,j}^2 \big ( (1+|dw|_g^2)^{-\frac{1}{2}} \big ) \\ 
&= -\big ( g^{ij} - (1+|dw|_g^2)^{-1} \, \partial^i w \, \partial^j w \big ) \, \big ( g^{kl} - (1+|dw|_g^2)^{-1} \, \partial^k w \, \partial^l w \big ) \\ 
&\hspace{10mm} \cdot (1+|dw|_g^2)^{-\frac{3}{2}} \, D_{i,k}^2 w \, D_{j,l}^2 w \\ 
&- (1+|dw|_g^2)^{-\frac{3}{2}} \, \text{\rm Ric}^{kl} \, \partial_k w \, \partial_l w \\ 
&- (1+|dw|_g^2)^{-\frac{3}{2}} \, \partial^k w \, \partial_k \big ( \big ( g^{ij} - (1+|dw|_g^2)^{-1} \, \partial^i w \, \partial^j w \big ) \, D_{i,j}^2 w \big ) 
\end{align*} 
together with (\ref{pde.for.w.2}), we obtain 
\begin{align*} 
&\big ( g^{ij} - (1+|dw|_g^2)^{-1} \, \partial^i w \, \partial^j w \big ) \, D_{i,j}^2 \big ( (1+|dw|_g^2)^{-\frac{1}{2}} \big ) \\ 
&+ (1+|dw|_g^2)^{-\frac{1}{2}} \, S^{ij} \, \partial_i w \, \partial_j \big ( (1+|dw|_g^2)^{-\frac{1}{2}} \big ) \\ 
&= -\big ( g^{ij} - (1+|dw|_g^2)^{-1} \, \partial^i w \, \partial^j w \big ) \, \big ( g^{kl} - (1+|dw|_g^2)^{-1} \, \partial^k w \, \partial^l w \big ) \\ 
&\hspace{10mm} \cdot (1+|dw|_g^2)^{-\frac{3}{2}} \, D_{i,k}^2 w \, D_{j,l}^2 w \\ 
&- (1+|dw|_g^2)^{-\frac{3}{2}} \, \text{\rm Ric}^{kl} \, \partial_k w \, \partial_l w \\ 
&- (1+|dw|_g^2)^{-1} \, D^k q^{ij} \, \big ( g_{ij} - (1+|dw|_g^2)^{-1} \, \partial_i w \, \partial_j w \big ) \, \partial_k w \\ 
&- (1+|dw|_g^2)^{-1} \, (|\zeta \, dw + w \, d\zeta|_g^2 - w^2 \, |d\zeta|_g^2) 
\end{align*} 
at each point in $\Omega$. Using the estimate $\text{\rm Ric} \geq -A \, \sigma^{-2} \, g$, we conclude that  
\begin{align*} 
&\Lambda \, \big ( g^{ij} - (1+|dw|_g^2)^{-1} \, \partial^i w \, \partial^j w \big ) \, D_{i,j}^2 \big ( (1+|dw|_g^2)^{-\frac{1}{2}} \big ) \\ 
&+ \Lambda \, (1+|dw|_g^2)^{-\frac{1}{2}} \, S^{ij} \, \partial_i w \, \partial_j \big ( (1+|dw|_g^2)^{-\frac{1}{2}} \big ) \\ 
&\leq \Lambda \, A \, \sigma^{-2} \, (1+|dw|_g^2)^{-\frac{3}{2}} \, |dw|_g^2 \\ 
&+ \Lambda \, (1+|dw|_g^2)^{-1} \, n \, |Dq|_g \, |dw|_g \\ 
&+ \Lambda \, (1+|dw|_g^2)^{-1} \, w^2 \, |d\zeta|_g^2 \\ 
&\leq A \, \sigma^{-2} \, e^{A^2 \sigma^{-1} (w-\psi)} \, (1+|dw|_g^2)^{-1} \, |dw|_g^2 \\ 
&+ e^{A^2 \sigma^{-1} (w-\psi)} \, (1+|dw|_g^2)^{-\frac{1}{2}} \, n \, |Dq|_g \, |dw|_g \\ 
&+ e^{A^2 \sigma^{-1} (w-\psi)} \, (1+|dw|_g^2)^{-\frac{1}{2}} \, w^2 \, |d\zeta|_g^2 
\end{align*} 
at the point $\bar{x}$. On the other hand, using the identity (\ref{pde.for.w.2}), we obtain 
\begin{align*} 
&\big ( g^{ij} - (1+|dw|_g^2)^{-1} \, \partial^i w \, \partial^j w \big ) \, D_{i,j}^2(e^{A^2 \sigma^{-1} (w-\psi)}) \\ 
&+ (1+|dw|_g^2)^{-\frac{1}{2}} \, S^{ij} \, \partial_i w \, \partial_j(e^{A^2 \sigma^{-1} (w-\psi)}) \\ 
&= A^4 \, \sigma^{-2} \, e^{A^2 \sigma^{-1} (w-\psi)} \, \big ( g^{ij} - (1+|dw|_g^2)^{-1} \, \partial^i w \, \partial^j w \big ) \, \partial_i(w-\psi) \, \partial_j(w-\psi) \\ 
&+ A^2 \, \sigma^{-1} \, e^{A^2 \sigma^{-1} (w-\psi)} \, \big ( g^{ij} - (1+|dw|_g^2)^{-1} \, \partial^i w \, \partial^j w \big ) \, D_{i,j}^2 (w-\psi) \\ 
&+ A^2 \, \sigma^{-1} \, e^{A^2 \sigma^{-1} (w-\psi)} \, (1+|dw|_g^2)^{-\frac{1}{2}} \, S^{ij} \, \partial_i w \, \partial_j(w-\psi) \\ 
&= A^4 \, \sigma^{-2} \, e^{A^2 \sigma^{-1} (w-\psi)} \, (1+|dw|_g^2)^{-1} \, |dw|_g^2 \\ 
&- 2A^4 \, \sigma^{-2} \, e^{A^2 \sigma^{-1} (w-\psi)} \, (1+|dw|_g^2)^{-1} \, \langle dw,d\psi \rangle_g \\ 
&+ A^4 \, \sigma^{-2} \, e^{A^2 \sigma^{-1} (w-\psi)} \, \big ( g^{ij} - (1+|dw|_g^2)^{-1} \, \partial^i w \, \partial^j w \big ) \, \partial_i \psi \, \partial_j \psi \\ 
&- A^2 \, \sigma^{-1} \, e^{A^2 \sigma^{-1} (w-\psi)} \, \big ( g^{ij} - (1+|dw|_g^2)^{-1} \, \partial^i w \, \partial^j w \big ) \, D_{i,j}^2 \psi \\ 
&- 2A^2 \, \sigma^{-1} \, e^{A^2 \sigma^{-1} (w-\psi)} \, (1+|dw|_g^2)^{-\frac{1}{2}} \, q^{ij} \, \partial_i w \, \partial_j \psi \\ 
&+ A^2 \, \sigma^{-1} \, e^{A^2 \sigma^{-1} (w-\psi)} \, (1+|dw|_g^2)^{-\frac{3}{2}} \, q^{ij} \, \partial_i w \, \partial_j w \, (1 + \langle dw,d\psi \rangle_g) \\ 
&+ A^2 \, \sigma^{-1} \, e^{A^2 \sigma^{-1} (w-\psi)} \, (1+|dw|_g^2)^{-\frac{1}{2}} \, (\text{\rm tr}_g(q) + \zeta^2 \, w) \, (1 + \langle dw,d\psi \rangle_g) 
\end{align*} 
at each point in $\Omega$. Consequently, 
\begin{align*} 
&\big ( g^{ij} - (1+|dw|_g^2)^{-1} \, \partial^i w \, \partial^j w \big ) \, D_{i,j}^2(e^{A^2 \sigma^{-1} (w-\psi)}) \\ 
&+ (1+|dw|_g^2)^{-\frac{1}{2}} \, S^{ij} \, \partial_i w \, \partial_j(e^{A^2 \sigma^{-1} (w-\psi)}) \\ 
&\geq A^4 \, \sigma^{-2} \, e^{A^2 \sigma^{-1} (w-\psi)} \, (1+|dw|_g^2)^{-1} \, |dw|_g^2 \\ 
&- 2A^4 \, \sigma^{-2} \, e^{A^2 \sigma^{-1} (w-\psi)} \, (1+|dw|_g^2)^{-1} \, |dw|_g \, |d\psi|_g \\ 
&- A^2 \, \sigma^{-1} \, e^{A^2 \sigma^{-1} (w-\psi)} \, n \, |D^2 \psi|_g \\ 
&- A^2 \, \sigma^{-1} \, e^{A^2 \sigma^{-1} (w-\psi)} \, (1+|dw|_g^2)^{-\frac{1}{2}} \, (2n \, |q|_g + \zeta^2 \, |w|) \, (1+ |dw|_g \, |d\psi|_g) 
\end{align*} 
at each point in $\Omega$. Putting these facts together, we obtain 
\begin{align*} 
&A^4 \, |dw|_g^2 - 2A^4 \, |dw|_g \, |d\psi|_g - A^2 \, \sigma \, (1+|dw|_g^2) \, n \, |D^2 \psi|_g \\ 
&- A^2 \, \sigma \, (1+|dw|_g^2)^{\frac{1}{2}} \, (2n \, |q|_g + \zeta^2 \, |w|) \, (1+|dw|_g \, |d\psi|_g) \\ 
&\leq A \, |dw|_g^2 + \sigma^2 \, (1+|dw|_g^2)^{\frac{1}{2}} \, n \, |Dq|_g \, |dw|_g \\ 
&+ \sigma^2 \, (1+|dw|_g^2)^{\frac{1}{2}} \, w^2 \, |d\zeta|_g^2 
\end{align*} 
at the point $\bar{x}$. In order to estimate the terms on the left hand side, we use the inequalities $|d\psi|_g \leq \frac{1}{4} \, A$, $n \, |D^2 \psi|_g \leq \frac{1}{4} \, A \, \sigma^{-1}$, $n \, |q|_g \leq A \, \sigma^{-1}$, and $\zeta^2 \, |w| \leq A \, \sigma^{-1}$. In order to estimate the terms on the right hand side, we use the inequalities $n \, |Dq|_g \leq A \, \sigma^{-2}$, $|w| \leq A \, \sigma$, and $|d\zeta|_g^2 \, |w| \leq A \, \sigma^{-3}$. Thus, we conclude that 
\begin{align*} 
&A^4 \, |dw|_g^2 - \frac{1}{2} \, A^5 \, |dw|_g - \frac{1}{4} \, A^3 \, (1+|dw|_g^2) - 3A^3 \, (1+|dw|_g^2)^{\frac{1}{2}} \, (1 + \frac{1}{4} \, A \, |dw|_g) \\ 
&\leq A \, |dw|_g^2 + A \, (1+|dw|_g^2)^{\frac{1}{2}} \, |dw|_g + A^2 \, (1+|dw|_g^2)^{\frac{1}{2}} 
\end{align*}
at the point $\bar{x}$. Since $\frac{1}{4} \, A^4 - \frac{1}{4} \, A^3 > 2A$, we can bound $|dw(\bar{x})|_g$ from above by some function of $A$. Since $|\psi| \leq A \, \sigma$ and $|w| \leq A \, \sigma$, we conclude that  
\[(e^{A^2 \sigma^{-1} (w-\psi)} - 1) \, (1+|dw|_g^2)^{\frac{1}{2}} \leq e^{2A^3} \, (1+|dw|_g^2)^{\frac{1}{2}} \leq C(A)\] 
at the point $\bar{x}$. Thus, $\Lambda \leq C(A)$, as claimed. This completes the proof of Proposition \ref{interior.gradient.estimate}. \\

\section{The Schoen-Yau identity for the modified Jang equation}

In this section, we state a crucial identity which originated in the groundbreaking work of Schoen and Yau \cite{Schoen-Yau3}. The following identity is a straightforward generalization of the identity proved in the work of Jang \cite{Jang} and Schoen and Yau \cite{Schoen-Yau3}. 

\begin{proposition}[cf. P.S.~Jang \cite{Jang}; R.~Schoen, S.T.~Yau \cite{Schoen-Yau3}]
\label{Schoen.Yau.identity}
Let $(M,g)$ be a Riemannian manifold. Suppose that $w$ and $\Theta$ are smooth functions on $M$ such that 
\[g^{ik} \, g^{jl} \, \big ( g_{ij} - (1+|dw|_g^2)^{-1} \, \partial_i w \, \partial_j w \big ) \, \big ( (1+|dw|_g^2)^{-\frac{1}{2}} \, D_{k,l}^2 w - q_{kl} \big ) = \Theta\] 
at each point in $M$. We define a Riemannian metric $\check{g}$ and a one-form $\Xi$ by 
\[\check{g}_{ij} = g_{ij} + \partial_i w \, \partial_j w\] 
and 
\[\Xi_k = \frac{1}{2} \, \partial_k \log (1+|dw|_g^2) - (1+|dw|_g^2)^{-\frac{1}{2}} \, g^{ij} \, q_{ik} \, \partial_j w.\] 
Then 
\begin{align*} 
&\frac{1}{2} \, R_{\check{g}} - |\Xi|_{\check{g}}^2 + \text{\rm div}_{\check{g}} \Xi \\ 
&= \frac{1}{2} \, \check{g}^{ik} \, \check{g}^{jl} \, \big ( (1+|dw|_g^2)^{-\frac{1}{2}} \, D_{i,j}^2 w - q_{ij} \big ) \, \big ( (1+|dw|_g^2)^{-\frac{1}{2}} \, D_{k,l}^2 w - q_{kl} \big ) \\ 
&+ \mu - (1+|dw|_g^2)^{-\frac{1}{2}} \, \langle dw,J \rangle_g \\ 
&+ (1+|dw|_g^2)^{-\frac{1}{2}} \, \langle dw,d\Theta \rangle_g + \frac{1}{2} \, \Theta^2 + \Theta \, \text{\rm tr}_{\check{g}}(q) 
\end{align*}
at each point in $M$.
\end{proposition} 

\textbf{Proof.} 
This identity can be derived by applying Proposition 7.32 in \cite{Lee} to the graph of $w$ in $M \times \mathbb{R}$. We refer to Subsection 3.6 in \cite{Andersson-Eichmair-Metzger} and Section 2 in \cite{Eichmair3} for further discussion. \\

\begin{remark} 
In the work of Schoen and Yau \cite{Schoen-Yau3}, the identity in Proposition \ref{Schoen.Yau.identity} is applied with $\Theta=0$. A more general version was considered in Eichmair's work (see \cite{Eichmair1}, p.~569--570).
\end{remark}

\end{document}